\documentclass[12pt]{amsart}
\usepackage{latexsym, amsfonts, amsmath, amssymb, amsthm, cite}

\newtheorem{lemma}{Lemma}
\newtheorem{proposition}{Proposition}
\newtheorem{theorem}{Theorem}

\usepackage{graphicx}
\usepackage{mathrsfs}

\date{\today }
 
\title[Lower bounds on the distribution of central values]{Conditional lower bounds on the distribution of central values in families of $L$-functions}
\author{Maksym Radziwi{\l}{\l} and Kannan Soundararajan} 
 \address{Department of Mathematics \\ University of Texas at Austin \\ 2515 Speedway, Stop C1200 \\ Austin \\ TX 78712} 
\email{maksym.radziwill@gmail.com}
\address{Department of Mathematics \\ Stanford University \\
450 Jane Stanford Way,  Bldg. 380\\ Stanford \\ CA 94305-2125}
\email{ksound@stanford.edu}
\thanks{The first author was partially supported by DMS-1902063. The second author is partially supported by an NSF grant, and a Simons Investigator award from the Simons Foundation} 
  \dedicatory{To Henryk Iwaniec, with admiration}

\begin{document}

\begin{abstract}  We establish a general principle that any lower bound on the non-vanishing of central $L$-values obtained through studying the one-level density 
of low-lying zeros can be refined to show that most such $L$-values have the typical size conjectured by Keating and Snaith.  
We illustrate this technique in the case of quadratic twists of a given elliptic curve, and similar results would hold for the many examples studied by Iwaniec, Luo, and Sarnak  in their pioneering work on $1$-level densities \cite{ILS}.
\end{abstract}

\maketitle

\section{Introduction}  

\noindent Selberg \cite{Se1, Se2} (see \cite{RaSo2} for a recent treatment) established that if $t$ is chosen uniformly from $[0,T]$ then the values $\log |\zeta(\tfrac 12+it)|$ are distributed approximately like a Gaussian random variable with mean $0$ and variance $\tfrac 12\log \log T$.   More recently, Keating and Snaith \cite{KeSn} have conjectured that central values in families of $L$-functions have an analogous log-normal distribution with a prescribed mean and variance depending on the ``symmetry type" of the family.  This is a powerful conjecture which gives more precise versions of conjectures on the non-vanishing of $L$-values; for example, it refines Goldfeld's conjecture (towards which remarkable progress has been made with the work of Smith \cite{smith}) that the rank in families of quadratic twists of an elliptic curve is $0$ for almost all twists with even sign of the functional equation.  In \cite{RaSo1} we enunciated a general principle which shows the upper bound (in a sense to be made precise below) part of the Keating--Saith conjecture in any family where somewhat more than the first moment can be computed.  In this paper, we consider the complementary problem of obtaining lower bounds in the Keating-Saith conjecture, which is intimately tied up with questions on the non-vanishing of $L$-values.   One analytic approach, conditional on the Generalized Riemann Hypothesis, towards such non-vanishing results is based on 
computing the $1$-level density for low lying zeros in families of $L$-functions, and our goal in this paper is to show how this approach (in the situations where it succeeds in producing a positive proportion of non-vanishing) may be refined to give corresponding lower bounds towards the Keating--Snaith conjectures.    In a later paper, we shall consider similar refinements of the mollifier method, which is another analytic approach that in many cases establishes non-vanishing results unconditionally.  Algebraic approaches such as Smith's work \cite{smith} on Goldfeld's conjecture are capable of establishing definitive non-vanishing results (or, for other examples, see Rohrlich \cite{Ro1, Ro2} and Chinta \cite{Chi}), but we are unable to refine these methods to show that the non-zero values that are produced in fact have the typical size predicted by the Keating-Snaith conjectures.  

To illustrate our method, we treat the family of quadratic twists of an elliptic curve $E$ defined over ${\Bbb Q}$ with conductor 
$N$, where the $1$-level density of low lying zeros has been studied by many authors, notably Heath-Brown \cite{HB}.   Let the associated $L$-function be 
$$ 
L(s,E) =\sum_{n=1}^{\infty} a(n) n^{-s}, 
$$ 
where the coefficients $a(n)$ are normalized such that $|a(n)| \le d(n)$.  Since elliptic curves are known to be modular, $L(s,E)$ has an analytic continuation to the entire 
complex plane and satisfies the functional equation 
$$ 
\Lambda(s,E) = \epsilon_E \Lambda(1-s,E), 
$$ 
where $\epsilon_E$, the root number, is $\pm 1$ and 
$$ 
\Lambda(s,E) = \Big(\frac{\sqrt{N}}{2\pi }\Big)^s \Gamma(s+\tfrac 12)L(s,E). 
$$ 
Throughout the paper, let $d$ denote a fundamental discriminant coprime to $2N$, and let $\chi_d=(\frac{d}{\cdot})$ denote the associated primitive quadratic character.  Let $E_d$ denote the quadratic twist of $E$ by $d$, and let its associated $L$-function be 
$$ 
L(s,E_d) = \sum_{n=1}^{\infty} a(n) \chi_d(n)n^{-s}. 
$$ 
If $(d,N)=1$ then $E_d$ has conductor $Nd^2$ and the completed $L$-function 
$$ 
\Lambda(s,E_d)= \Big(\frac{\sqrt{N}|d|}{2\pi}\Big)^{s} \Gamma(s+\tfrac 12) L(s,E_d)
$$ 
is entire and satisfies the functional equation 
$$ 
\Lambda(s,E_d)= \epsilon_E(d) \Lambda(1-s,E_d) 
$$ 
with 
$$ 
\epsilon_E(d) = \epsilon_E \chi_d(-N). 
$$ 
Note that, by Waldspurger's theorem, $L(\tfrac 12, E_d) \ge 0$. Of course $L(\tfrac 12, E_d)=0$ when 
$\epsilon_E(d)=-1$, and in this paper, we shall restrict attention to those twists with root number $1$.  Put therefore
$$ 
{\mathcal E} = \{ d: \ \ d \text{ a fundamental discriminant with } (d, 2N)=1 \text{ and } \epsilon_{E}(d)=1\}. 
$$ 

The Keating-Snaith conjectures predict that for $d\in {\mathcal E}$, the quantity $\log L(\tfrac 12,E_d)$ has an 
approximately normal distribution with mean $-\tfrac 12 \log \log |d|$ and variance $\log \log |d|$.   To state this precisely, let $\alpha < \beta$ be real numbers, 
and for any $X\ge 20$, let us define 
\begin{equation}
\label{1.1}  
{\mathcal N}(X;\alpha, \beta) = \Big| \Big\{ d\in {\mathcal E}, X < |d|\le 2X:  \ \frac{\log L(\tfrac 12,E_d)+\tfrac 12 \log \log |d|}{\sqrt{\log \log |d|}} \in (\alpha, \beta) \Big\} \Big|. 
\end{equation} 
Then the Keating-Snaith conjecture states that, for fixed intervals $(\alpha, \beta)$ and as $X\to \infty$,  
\begin{equation}
\label{1.2} 
{\mathcal N}(X;\alpha, \beta) = | \{d\in {\mathcal E}, X\le |d|\le 2X\} | \Big(\frac{1}{\sqrt{2\pi}} \int_{\alpha}^{\beta} e^{-\frac{x^2}{2} } dx + o(1) \Big). 
\end{equation}  
Here we interpret $\log L(\tfrac 12,E_d)$ to be negative infinity if $L(\tfrac 12,E_d)=0$, and the conjecture implies in particular that $L(\tfrac 12, E_d) \neq 0$ for almost all $d\in {\mathcal E}$.   Towards this conjecture, we established in \cite{RaSo1} that ${\mathcal N}(X;\alpha,\infty)$ is bounded above by the right hand side of the conjectured relation \eqref{1.2}.  Complementing this, we now establish a conditional lower bound for ${\mathcal N}(X;\alpha,\beta)$.  

\begin{theorem}  Assume the Generalized Riemann Hypothesis for the family of twisted $L$-functions $L(s,E \times \chi)$ for all Dirichlet characters $\chi$.  Then for fixed intervals $(\alpha,\beta)$ and as $X\to \infty$ we have 
$$ 
{\mathcal N}(X;\alpha,\beta) \ge  | \{d\in {\mathcal E}, X\le |d|\le 2 X\} | \Big( \frac 14 \frac{1}{\sqrt{2\pi}} \int_{\alpha}^{\beta} e^{-\frac{x^2}{2} } dx + o(1) \Big). 
$$
\end{theorem}  

Above we have assumed GRH for all character twists of $L(s,E)$; this is largely for convenience, and would allow us to restrict $d$ in progressions.  With more effort one could relax the assumption to GRH for the family of quadratic twists $L(s,E_d)$.  Note that the factor $\frac 14$ in our theorem matches the proportion of quadratic twists with non-zero $L$-value obtained in Heath-Brown's work \cite{HB}.

While we have described results for the family of quadratic twists of an elliptic curve, the method is very general and applies to many situations where $1$-level densities 
of low lying zeros in families have been analyzed and yield a positive proportion of non-vanishing for the central values.  The work of Iwaniec, Luo, and Sarnak  \cite{ILS} gives many such examples, and the technique described here refines their non-vanishing corollaries, showing that the non-zero $L$-values that are produced have the typical size conjectured by Keating and Snaith.   For instance, consider the family of symmetric square $L$-functions $L(s,\text{sym}^2 f)$ where $f$ ranges over Hecke eigenforms of weight $k$ for the full modular group (denote the set of such eigenforms by $H_k$), with $k\le K$ (thus there are about $K^2/48$ such $L$-values).  Assuming GRH in this family, Iwaniec, Luo, and Sarnak (see Corollary 1.8 of \cite{ILS}), showed that at least a proportion $\frac 89$ of these $L$-values are non-zero.  We may refine this to say that for any fixed interval $(\alpha, \beta)$ and as $K \to \infty$ 
$$ 
\Big| \{ f\in H_k, \ \ k \le K:  \frac{\log L(\tfrac 12, \text{sym}^2 f) - \frac 12 \log \log k}{\sqrt{\log \log k}} \in (\alpha, \beta)\Big\} \Big| 
\ge \Big(\frac 89 \frac{1}{\sqrt{2\pi }} \int_{\alpha}^{\beta} e^{-x^2/2} dx + o(1)\Big) \frac{K^2}{48}. 
$$ 

We end the introduction by mentioning the recent work of Bui, Evans, Lester, and Pratt \cite{BELP} who establish ``weighted'' (where the weight is a 
mollified central value) analogues of the Keating-Snaith conjecture.  This amounts to a form of conditioning on non-zero value since central values that are zero are assigned a weight equal to zero. The use of such a weighted measure allows \cite{BELP} to establish a full asymptotic, however as a side effect they have little control over the nature of the weight.

\noindent{\bf Acknowledgments.}   We are grateful to Emmanuel Kowalski for a careful reading of the paper, and helpful comments.  
The first author was partially supported by DMS-1902063. The second author is partially supported by an NSF grant, and a Simons Investigator award from the Simons Foundation.  
The paper was completed while KS was a Senior Fellow at the Institute for Theoretical Studies, ETH Z{\" u}rich, whom he thanks for their excellent working conditions, and warm hospitality. 

\section{Notation and statements of the key propositions}  

We begin by introducing some notation, as in our paper \cite{RaSo1}, and then describing three key propositions which underlie the proof 
of the main theorem.  Let $N_0$ denote the lcm of $8$ and $N$.  Let $\kappa$ be $\pm 1$, and let $a \bmod {N_0}$ denote a residue class with 
$a\equiv 1$ or $5 \bmod 8$.   We assume that $\kappa$ and $a$ are such that for any fundamental discriminant $d$ with sign $\kappa$ and 
with $d\equiv a \bmod {N_0}$, the root number $\epsilon_E(d) = \epsilon_E \chi_d(-N)$ equals $1$.  Define 
$$ 
{\mathcal E}(\kappa, a) = \{ d\in {\mathcal E}: \ \ \kappa d >0, \ \ d\equiv a \bmod{N_0} \} ,
$$ 
so that ${\mathcal E}$ is the union of all such sets ${\mathcal E}(\kappa, a)$.  

We write below 
$$ 
-\frac{L^{\prime}}{L}(s, E)  = \sum_{n=1}^{\infty} \frac{\Lambda_E(n)}{n^s}, 
$$ 
where $|\Lambda_E(n)| \le 2\Lambda(n)$ so that $\Lambda_E(n) = 0$ unless $n = p^k$ is a prime power.  
If $p\nmid N_0$, we may write $a(p) = \alpha_p+ \overline{\alpha_p}$ for a complex number $\alpha_p$ of magnitude $1$ (unique up to complex conjugation),  and 
then 
$$ 
\Lambda_E(p^k) = (\alpha_p^k + \overline{\alpha_p}^k) \log p. 
$$ 
Note that 
$$ 
-\frac{L^{\prime}}{L}(s,E_d) = \sum_{n=1}^{\infty} \frac{\Lambda_E(n)}{n^s} \chi_d(n). 
$$ 
For fundamental discriminants $d \in {\mathcal E}$ with $|d| \le 3X$,  and a parameter $3 \le x $ define 
\begin{equation} 
\label{2.1} 
{\mathcal P}(d;x)  = \sum_{ \substack{p \le x \\ p\nmid N_0} } \frac{a(p)}{\sqrt{p}} \chi_d(p). 
\end{equation}

Let $h$ denote a smooth function with compactly supported Fourier transform 
$$ 
{\widehat h}(\xi) = \int_{-\infty}^{\infty} h(t)e^{-2\pi i \xi t} dt,  
$$ 
and such that $|h(x)| \ll (1+x^2)^{-1}$ for all $x \in {\mathbb R}$.  For concreteness, one could simply consider 
$h$ to be the Fejer kernel given by 
\begin{equation} 
\label{2.1a} 
h(x) = \Big( \frac{\sin(\pi x)}{\pi x}\Big)^2, \qquad {\widehat h}(t) = \max(1-|t|, 0). 
\end{equation} 

Lastly, let $\Phi$ denote a smooth, non-negative function compactly supported in $[\frac 12, \frac 52]$ with 
$\Phi(x)= 1$ for $x \in [1,2]$, and we put ${\check \Phi}(s) = \int_0^{\infty} \Phi(x) x^s dx$.  
Below all implied constants will be allowed to depend on $N$, $h$, and $\Phi$, which are considered fixed.  

Our first proposition connects $\log L(\frac 12, E_d)$ with the sum over primes ${\mathcal P}(d;x)$ 
(for suitable $x$) with an error term given in terms of the zeros of $L(s,E_d)$.  Such formulae have a long 
history, going back to Selberg, and the work here complements an upper bound version that played a key 
role in \cite{Sou1}.  

\begin{proposition}  \label{prop1} Let $d$ be a fundamental discriminant in ${\mathcal E}$, and let $3 \le x \le |d|$.  Assume GRH for $L(s,E_d)$, and suppose 
that $L(\tfrac 12, E_d)$ is not zero.   Let $\gamma_d$ run over the ordinates of the non-trivial zeros of $L(s,E_d)$.    
Then 
$$ 
\log L(\tfrac 12, E_d) =  {\mathcal P}(d;x) - \tfrac 12 \log \log x +O\Big( \frac{\log |d|}{\log x} + \sum_{\gamma_d} \log \Big(1 + \frac{1}{(\gamma_d \log x)^2}\Big) \Big). 
$$ 
\end{proposition} 

To analyze sums over the zeros we shall use the following proposition, whose proof is based on the explicit formula.  
The ideas behind this proposition are also familiar, and in this setting (and in the case $\ell =1$ below) may be traced back to the work of Heath-Brown \cite{HB}.    

\begin{proposition}  \label{prop2}  Let $h$ be a smooth function with $h(x)\ll (1+x^2)^{-1}$ and  whose Fourier transform is compactly supported in $[-1,1]$.  
Let $L\ge 1$ be a real number, and $\ell$ be a positive integer coprime to $N_0$, and assume that $e^{L} \ell^2 \le X^2$.  If $\ell$ is neither a square, nor a prime times a square, then  
\begin{equation} 
\label{2.11} 
\sum_{d \in {\mathcal E}(\kappa, a)} \Big(\sum_{\gamma_d} h\Big(\frac{\gamma_d L}{2\pi}\Big) \Big) \chi_d(\ell) \Phi\Big( \frac{\kappa d}{X}\Big) \ll X^{\frac 12+\epsilon} \ell^{\frac 12} e^{\frac L4}.   
\end{equation} 
If $\ell$ is a square then 
\begin{align} 
\label{2.12} 
\sum_{d \in {\mathcal E}(\kappa, a)} \Big(\sum_{\gamma_d} h\Big(\frac{\gamma_d L}{2\pi}\Big) \Big)& \chi_d(\ell) \Phi\Big( \frac{\kappa d}{X}\Big) 
= O(X^{\frac 12+\epsilon} \ell^{\frac 12} e^{\frac L4}) \nonumber \\ 
&+ \frac{X}{N_0} \prod_{p|\ell} \Big(1 +\frac 1p\Big)^{-1} \prod_{p\nmid N_0} \Big(1 -\frac 1{p^2}\Big) {\widehat \Phi}(0) \Big( \frac{2\log X}{L} {\widehat h}(0) + \frac{h(0)}{2} + 
O\Big( \frac 1L\Big) \Big). 
\end{align}
Finally if $\ell$ is $q$ times a square, for a prime number $q$, then 
\begin{equation} 
\label{2.13} 
\sum_{d \in {\mathcal E}(\kappa, a)} \Big(\sum_{\gamma_d} h\Big(\frac{\gamma_d L}{2\pi}\Big) \Big) \chi_d(\ell) \Phi\Big( \frac{\kappa d}{X}\Big)  
\ll \frac{X}{LN_0} \frac{\log q}{\sqrt{q}} \prod_{p|\ell} \Big(1+ \frac 1p\Big)^{-1} + X^{\frac 12+\epsilon} \ell^{\frac 12} e^{L/4}.
\end{equation} 
\end{proposition}  

Finally, to understand the distribution of ${\mathcal P}(d;x)$ both when $d$ is chosen uniformly over discriminants $d \in {\mathcal E}$, 
and when $d \in {\mathcal E}$ is weighted by contributions from low-lying zeros, we shall use the method of moments, drawing upon the following 
proposition. 

\begin{proposition} \label{prop3}  Let $k$ be any fixed non-negative integer.  Let $X$ be large, and put $x=X^{1/\log \log \log X}$.  Then 
\begin{equation} 
\label{2.2} 
\sum_{d\in {\mathcal E}(\kappa,a)} {\mathcal P}(d;x)^k \Phi\Big( \frac{\kappa d}{X}\Big) = \Big( \sum_{d \in {\mathcal E}(\kappa, a)} \Phi \Big( \frac{\kappa d}{X} \Big) 
\Big) (\log \log X)^{\frac k2} (M_k +o(1)), 
\end{equation} 
where $M_k$ denotes the $k$-th Gaussian moment: 
$$ 
M_k = \frac{1}{\sqrt{2\pi}} \int_{-\infty}^{\infty} x^k e^{-\frac{x^2}{2}} dx = \begin{cases} 
\frac{k!}{2^{k/2} (k/2)!} &\text{ if  } k \text{ is even}\\ 
0 &\text{if  } k \text{ is odd}.\\ 
\end{cases}
$$ 
Further, for any parameter $L\ge 1$ with $e^{L} \le X^2$ we have,  
\begin{align} 
\label{2.3} 
\sum_{d\in {\mathcal E}} {\mathcal P}(d;x)^k& \Big( \sum_{\gamma_d} h\Big( \frac{\gamma_d L}{2\pi } \Big) \Big) \Phi\Big( \frac{\kappa d}{X}\Big) 
= O(X^{\frac 12+\epsilon}e^{\frac L4}) \nonumber \\ 
&+\frac{X}{N_0} \prod_{p\nmid N_0} \Big( 1-\frac{1}{p^2}\Big) 
  {\widehat \Phi}(0) \Big( \frac{2\log X}{L} {\widehat h}(0)+ \frac{h(0)}{2}+O\Big( \frac 1L\Big)\Big) (M_k +o(1)) (\log \log X)^{\frac k2}.
\end{align}
\end{proposition}

\section{Deducing the Theorem from the  main propositions} 
 
We keep the notations introduced in Section 2.  Let $X$ be large, and put $x= X^{1/\log \log \log X}$.  

\begin{lemma} \label{lem3.1}  Let $\alpha <\beta$ be real numbers.   Let $\mathcal{G}_{X}(\alpha, \beta)$ denote the set of discriminants $d\in {\mathcal E}$ 
with $X \le |d| \le 2X$ such that 
$$ 
\frac{{\mathcal P}(d;x)}{\sqrt{\log \log X}} \in (\alpha, \beta), 
$$ 
and such that there are no zeros $\rho_d =\tfrac 12 +i\gamma_d$ of $L(s,E_d)$ with $|\gamma_d| \le (\log X \log \log X)^{-1}$.  
Then, for any $\delta >0$,  
$$ 
| \mathcal{G}_{X}(\alpha, \beta)| \ge \Big( \frac 14- \delta\Big) \Big( \frac{1}{\sqrt{2\pi}} \int_{\alpha}^{\beta} e^{-t^2/2} dt +o(1)\Big) |\{ d\in {\mathcal E}: X \le |d|\le 2X \}.  
 $$ 
 \end{lemma} 
 \begin{proof}   Take $\Phi$ to be a smooth approximation to the indicator function of the interval $[1,2]$, and let $\kappa$ and $a \bmod {N_0}$ be as in 
 Section 2.  The first part of Proposition \ref{prop3} (namely \eqref{2.2}) together with the method of moments shows that 
 \begin{equation} 
 \label{3.1} 
 \sum_{\substack{ d\in {\mathcal E}(\kappa, a) \\ {\mathcal P}(d;x)/\sqrt{\log \log X} \in (\alpha, \beta)} } 
 \Phi\Big(\frac{\kappa d}{X}\Big)  = \Big( \frac{1}{\sqrt{2\pi}} \int_{\alpha}^{\beta} e^{-t^2/2} dt +o(1)\Big) \Big( \sum_{d\in {\mathcal E}(\kappa, a)} \Phi \Big(\frac{\kappa d}{X}\Big)\Big). 
 \end{equation} 
 
 Next, take $h$ to be the Fejer kernel given in \eqref{2.1a}, and $L=(2-\delta/2) \log X$.  Then the second part of Proposition \ref{prop3} together with the method of moments 
 shows that 
 \begin{align*} 
  \sum_{\substack{ d\in {\mathcal E}(\kappa, a) \\ {\mathcal P}(d;x)/\sqrt{\log \log X} \in (\alpha, \beta)} }   &\sum_{\gamma_d} h\Big( \frac{\gamma_d L}{2\pi} \Big) 
 \Phi\Big(\frac{\kappa d}{X}\Big)
 = \Big( \frac{1}{\sqrt{2\pi}} \int_{\alpha}^{\beta} e^{-t^2/2} dt +o(1)\Big)  \sum_{d\in {\mathcal E}(\kappa, a)}     \sum_{\gamma_d} h\Big( \frac{\gamma_d L}{2\pi} \Big) \Phi \Big(\frac{\kappa d}{X}\Big) \\
 &= \Big( \frac{1}{\sqrt{2\pi}} \int_{\alpha}^{\beta} e^{-t^2/2} dt +o(1)\Big) \Big( \frac{1}{1-\delta/4} +\frac 12 +o(1)\Big)  \sum_{d\in {\mathcal E}(\kappa, a)} \Phi \Big(\frac{\kappa d}{X}\Big).
 \end{align*}
 Note that the weights $\sum_{\gamma_d} h(\gamma_d L/(2\pi))$ are always non-negative, and if $L(s,E_d)$ has a zero with $|\gamma_d| \le (\log X \log \log X)^{-1}$ 
 then the weight is $\ge 2+o(1)$ (since there would be a complex conjugate pair of such zeros, or a double zero at $\frac 12$).  Combining this with \eqref{3.1}, and summing over 
 all the possibilities for $\kappa$ and $a$, we obtain the lemma. 
  \end{proof} 
  
  \begin{lemma} \label{lem3.2} The number of discriminants $d\in {\mathcal E}$ with $X\le |d|\le 2X$ such that 
  $$ 
  \sum_{(\log X \log \log X)^{-1} \le |\gamma_d|} \log \Big( 1+ \frac {1}{(\gamma_d \log x)^2} \Big) \ge (\log \log \log X)^3
  $$ 
  is $\ll X/\log \log \log X$.  
  \end{lemma} 
  \begin{proof}  Applying Proposition \ref{prop2} with $\ell=1$, $h$ given as in \eqref{2.1a}, and $1\le L\le (2-\delta) \log X$, we obtain (after summing over the possibilities 
  for $\kappa$ and $a$)
 $$ 
  \sum_{\substack {d\in {\mathcal E} \\ X \le |d|\le 2X}} \sum_{\gamma_d} \Big( \frac{\sin(\gamma_d L/2)}{\gamma_d L/2}\Big)^2   \ll X \frac{\log X}{L}. 
$$ 
  Integrate both sides of this estimate over $L$ in the range $\log x \le L \le 2 \log x$.  Since, for any $y>0$ and $t \neq 0$, 
  $$ 
\frac{1}{y}  \int_{y}^{2y} \Big( \frac{\sin(\pi t u)}{\pi t u}\Big)^2 du \gg \min\Big( 1, \frac{1}{(ty)^2}\Big), 
$$ 
we obtain 
$$ 
\sum_{\substack {d\in {\mathcal E} \\ X \le |d|\le 2X}} \sum_{\gamma_d} \min \Big( 1, \frac{1}{(\gamma_d \log x)^2}\Big) \ll X \frac{\log X}{\log x} = X \log \log \log X. 
$$ 
Now if $|\gamma_d| \ge (\log X \log \log X)^{-1}$ then 
$$ 
\log \Big( 1+ \frac{1}{(\gamma_d \log x)^2}\Big) \ll (\log \log \log X) \min \Big(1 , \frac{1}{(\gamma_d \log x)^2}\Big), 
$$ 
and therefore we may conclude that 
$$ 
\sum_{\substack {d\in {\mathcal E} \\ X \le |d|\le 2X}}   \sum_{(\log X \log \log X)^{-1} \le |\gamma_d|} \log \Big( 1+ \frac {1}{(\gamma_d \log x)^2} \Big) 
\ll X (\log \log \log X)^2. 
$$ 
The lemma follows at once. 
  \end{proof}   
 
 With these results  in place, it is now a simple matter to deduce the main theorem.  By Proposition 1 \ref{prop1} we know that for $d \in {\mathcal E}$ with $X\le |d|\le 2X$
  $$
  \log L(\tfrac 12, E_d) = {\mathcal P}(d;x) - \tfrac 12 \log \log X + O(\log \log \log X) + 
  O\Big( \sum_{\gamma_d} \log \Big( 1+\frac{1}{(\gamma_d \log x)^2}\Big)\Big). 
  $$ 
  Lemma \ref{lem3.1} tells us that for $d\in \mathcal{G}_{X}(\alpha, \beta)$ we may arrange for ${\mathcal P}(d;x)/\sqrt{\log \log X}$ to lie in the interval $(\alpha, \beta)$ 
  and for there to be no zeros with $|\gamma_d| \le (\log X \log \log X)^{-1}$.   Lemma \ref{lem3.2} now allows us to discard $\ll X/\log \log \log X$ elements of 
  $\mathcal{G}_{X}(\alpha, \beta)$ so as to ensure that the contribution of zeros with $|\gamma_d|\ge (\log X\log \log X)^{-1}$ is $O((\log \log \log X)^3)$.  Thus there are 
  $$ 
  \ge  \Big( \frac 14- \delta\Big) \Big( \frac{1}{\sqrt{2\pi}} \int_{\alpha}^{\beta} e^{-t^2/2} dt +o(1)\Big) |\{ d\in {\mathcal E}: X \le |d|\le 2X \}, 
  $$ 
  fundamental discriminants $d\in {\mathcal E}$ with $X \le |d|\le 2X$ for which 
  $$ 
\frac{  \log L(\frac 12, E_d) + \frac 12\log \log X}{\sqrt{\log \log X}} + O\Big( \frac{(\log \log \log X)^3}{\sqrt{\log \log X}}\Big) \in (\alpha, \beta),
$$ 
which completes the proof.

 \section{Proof of Proposition \ref{prop1}} 
 
 A straight-forward adaptation of Lemma 1 from \cite{Sou1} (itself based on an identity of Selberg) shows that for any 
 $\sigma \ge \frac 12$ with $L(\sigma, E_d) \neq 0$, and any $x\ge 3$ one has  
 \begin{equation} 
 \label{4.1} 
 -\frac{L^{\prime}}{L}(\sigma, E_d) = \sum_{n\le x} \frac{\Lambda_E(n)}{n^{\sigma}} \chi_d(n) \frac{\log (x/n)}{\log x} + 
 \frac{1}{\log x} \Big( \frac{L^{\prime}}{L}\Big)^{\prime}(\sigma, E_d) + \frac{1}{\log x} \sum_{\rho_d} \frac{x^{\rho_d-\sigma}}{(\rho_d -\sigma)^2} 
 + O\Big( \frac{1}{x^{\sigma} \log x} \Big). 
 \end{equation} 
 Here $\rho_d$ runs over the non-trivial zeros of $L(s,E_d)$, and this identity in fact holds unconditionally.  
 
 Now assume GRH for $L(s,E_d)$ and write $\rho_d =\tfrac 12 +i\gamma_d$.  If $L(\tfrac 12, E_d) \neq 0$, then integrating both sides of 
 \eqref{4.1} from $\tfrac 12$ to $\infty$ yields 
 \begin{align}\label{4.2} 
 \log L(\tfrac 12, E_d) = \sum_{n\le x} \frac{\Lambda_E(n)}{\sqrt{n}\log n} \chi_d(n)& \frac{\log (x/n)}{\log x} - 
 \frac{1}{\log x} \frac{L^{\prime}}{L}(\tfrac 12, E_d)\nonumber \\
 &+ \frac{1}{\log x} \sum_{\gamma_d} \text{Re} \int_{\frac 12}^{\infty} \frac{x^{\rho_d-\sigma}}{(\rho_d-\sigma)^2} d\sigma 
 + O\Big( \frac{1}{\sqrt{x} (\log x)^2}\Big). 
 \end{align} 
 We may restrict attention to the real part of the integral above since all the other terms involved are real, or noting that the 
 zeros $\rho_d$ appear in conjugate pairs.
 
 Consider first the sum over $n$ in \eqref{4.2}.  The contribution from prime powers $n=p^k$ with $k\ge 3$ is plainly $O(1)$.  
 The contribution of the terms $n=p$ is ${\mathcal P}(d;x)+O(1)$, where the error term $O(1)$ arises from the primes dividing $N_0$.   Finally, by Rankin--Selberg theory 
 (see for instance \cite{IK}) it follows that 
 \begin{equation} 
 \label{4.3} 
 \sum_{\substack{ p\le y\\ p\nmid N_0}} \frac{(\alpha_p^2+\overline{\alpha_p}^2) \log p}{p} = \sum_{\substack{ p\le y \\ p\nmid N_0}} \frac{(a(p)^2-2)\log p}{p} = - \log y+ O(1), 
 \end{equation} 
 so that, by partial summation, 
 the contribution of the terms $n=p^2$ equals 
 $$
 \sum_{\substack{p\le \sqrt{x}\\ p\nmid N_0}} \frac{(\alpha_p^2 +\overline{\alpha_p}^2)}{2p} \frac{\log (x/p^2)}{\log x} +O(1) 
 = \sum_{\substack{ p\le \sqrt{x} \\ p\nmid N_0}} \frac{a(p)^2 -2}{2p} \frac{\log (x/p^2)}{\log x} +O(1)= -\frac 12 \log \log x + O(1). 
 $$ 
   Thus the contribution of the sum over $n$ in \eqref{4.2} is 
 \begin{equation} 
 \label{4.4} 
 {\mathcal P}(d;x) -\tfrac 12 \log \log x + O(1). 
 \end{equation}

 Next we turn to the sum over zeros in \eqref{4.2}.  If $|\gamma_d \log x| \ge 1$, then note that 
 $$ 
 \int_{\frac 12}^{\infty} \frac{x^{\rho_d -\sigma}}{(\rho_d-\sigma)^2} d\sigma 
 = O\Big( \frac{1}{\gamma_d^2} \int_{\frac 12}^{\infty} x^{\frac 12-\sigma} d\sigma \Big) 
 = O\Big( \frac{1}{\gamma_d^2 \log x}\Big) = O\Big( \log x \log \Big(1+ \frac{1}{\gamma_d^2 (\log x)^2}\Big)\Big). 
 $$ 
 If $|\gamma_d \log x| \le 1$, then we split into the ranges $\tfrac 12 \le \sigma \le \tfrac 12+\frac{1}{\log x}$ and 
 larger values of $\sigma$.  The first range contributes 
 \begin{align*} 
 \int_{\frac 12}^{\frac 12 + \frac{1}{\log x}} \text{Re} \frac{x^{\rho_d-\sigma}}{(\rho_d-\sigma)^2} d\sigma 
& = \int_{\frac 12}^{\frac 12 +\frac{1}{\log x}} \text{Re} \Big( \frac{1}{(\rho_d-\sigma)^2} + \frac{\log x}{(\rho_d-\sigma)} + O((\log x)^2)\Big) d\sigma 
 \\
 &= \text{Re} \Big(- \frac{1}{i\gamma_d}- \frac{1}{1/\log x -i\gamma_d} + \log x \log \frac{-i\gamma_d}{1/\log x -i\gamma_d} + O(\log x)\Big)\\ 
 &=O\Big( \log x \log \Big(1+ \frac{1}{\gamma_d^2 (\log x)^2}\Big)\Big),
 \end{align*}
 while the second range contributes 
 $$ 
\ll \int_{\frac 12+\frac{1}{\log x}}^{\infty} \frac{x^{\frac 12-\sigma}}{(\frac 12-\sigma)^2} d\sigma \ll \log x = O\Big( \log x \log \Big(1+ \frac{1}{\gamma_d^2 (\log x)^2}\Big)\Big).
$$ 
Thus in all cases the sum over zeros in \eqref{4.2} is 
\begin{equation} 
\label{4.5} 
O\Big( \log x \log \Big(1+ \frac{1}{\gamma_d^2 (\log x)^2}\Big)\Big).
\end{equation}

 Finally, taking logarithmic derivatives in the functional equation we find that 
 $$ 
 \frac{L^{\prime}}{L} (\tfrac 12, E_d) = -\log (\sqrt{N} |d|) + O(1). 
 $$ 
 The proposition follows upon combining this with \eqref{4.2},  \eqref{4.4}, and \eqref{4.5}. 
 
\section{Proof of Proposition \ref{prop2}} 

The proof of Proposition \ref{prop2} is based on the explicit formula, which we first recall in our context.  

\begin{lemma} 
\label{expformula}   Let $h$ be a function with $h(x) \ll (1+x^2)^{-1}$ and with compactly supported Fourier transform ${\widehat h}(\xi) = 
\int_{-\infty}^{\infty} h(t) e^{-2\pi i\xi t} dt$.   Then, for any fundamental discriminant $d \in {\mathcal E}$ 
\begin{align*} 
\sum_{\gamma_d} h\Big(\frac{\gamma_d}{2\pi} \Big) =\frac{1}{2\pi}   \int_{-\infty}^{\infty} h\Big( \frac{t}{2\pi} \Big)& \Big( \log \frac{Nd^2}{(2\pi)^2} + 2\text{Re} \frac{\Gamma^{\prime}}{\Gamma} (1+it) \Big) dt   \\
&- \sum_{n} \frac{\Lambda_E(n)}{\sqrt{n}} \chi_{d}(n) \Big( {\widehat h}(\log n) + {\widehat h}(-\log n)\Big), 
\end{align*}
where the sum is over all ordinates of non-trivial zeros $1/2+ i\gamma_d$ of $L(s,E_d)$.   
\end{lemma}  

Applying the explicit formula to the dilated function $h_L(x) = h(xL)$ whose Fourier transform is $\frac 1L {\widehat h}(x/L)$, we 
obtain 
\begin{align} 
\label{5.1} 
\sum_{\gamma_d} h\Big(\frac{\gamma_d L}{2\pi} \Big) =\frac{1}{2\pi}   \int_{-\infty}^{\infty} h\Big( \frac{tL}{2\pi} \Big)& \Big( \log \frac{Nd^2}{(2\pi)^2} + 2\text{Re} \frac{\Gamma^{\prime}}{\Gamma} (1+it) \Big) dt \nonumber  \\
&- \frac{1}{L} \sum_{n} \frac{\Lambda_E(n)}{\sqrt{n}} \chi_{d}(n) \Big( {\widehat h}\Big(\frac{\log n}{L}\Big) + {\widehat h}\Big(-\frac{\log n}{L}\Big)\Big).
\end{align}
We multiply this expression by $\chi_d(\ell)$ and sum over $d$ with suitable weights.  Thus we find 
\begin{equation} 
\label{5.2} 
\sum_{d\in {\mathcal E}(\kappa, a)} \sum_{\gamma_d} h\Big(\frac{\gamma_d L}{2\pi} \Big) \chi_d(\ell) \Phi\Big( \frac{\kappa d}{X}\Big) = S_1 - S_2, 
\end{equation} 
where 
\begin{equation} 
\label{5.3} 
S_1 = \frac{1}{2\pi} \int_{-\infty}^{\infty} h\Big( \frac{tL}{2\pi}\Big) \sum_{d\in {\mathcal E}(\kappa,a)} \chi_d(\ell) \Big( \log \frac{Nd^2}{(2\pi)^2} + 2\text{Re }\frac{\Gamma^{\prime}}{\Gamma}(1+it)\Big) \Phi\Big(\frac{\kappa d}{X}\Big)dt, 
\end{equation} 
and 
\begin{equation} 
\label{5.4} 
S_2 = \frac{1}{L} \sum_{n} \frac{\Lambda_E(n)}{\sqrt{n}} \Big( {\widehat h}\Big(\frac{\log n}{L}\Big) + {\widehat h}\Big(-\frac{\log n}{L}\Big)\Big) \sum_{d\in {\mathcal E}(\kappa,a)} \chi_d(\ell n) 
\Phi\Big( \frac{\kappa d}{X}\Big). 
\end{equation}

The term $S_1$ is relatively easy to handle.  If $\ell$ is a square, it amounts to counting square-free integers $d$ lying in a suitable progression $\bmod {N_0}$ and coprime to $\ell$.  While if $\ell$ is not a square, the resulting sum is a non-trivial character sum, which exhibits substantial cancellation.   A more general term of this type is handled in Proposition 1 of \cite{RaSo1}, which we refer to for a detailed proof.  Thus when $\ell$ is not a square we find 
\begin{equation} 
\label{5.5} 
S_1 = O(X^{\frac 12+\epsilon} \sqrt{\ell}), 
\end{equation} 
while if $\ell$ is a square
\begin{align} 
\label{5.6} 
S_1 &= \frac{X}{N_0} \prod_{p|\ell} \Big(1+\frac 1p\Big)^{-1}\prod_{p\nmid N_0} \Big(1-\frac{1}{p^2}\Big) {\check \Phi}(0) (2\log X+ O(1)) \frac{{\widehat h}(0)}{L} 
+ O(X^{\frac 12+\epsilon} \sqrt{\ell}). 
\end{align} 

We now turn to the more difficult term $S_2$.  First we dispose of terms $n$ (which we may suppose is a prime power) 
that have a common factor with $N_0$.  Note that since $d$ is fixed in a residue class $\bmod {N_0}$, if $n$ is the power of a prime dividing $N_0$ then $\chi_d(n)$ is determined by 
the congruence condition on $d$.   Thus the contribution of these terms is 
\begin{equation} 
\label{5.7} 
\ll \frac{1}{L} \sum_{(n,N_0) >1} \frac{\Lambda(n)}{\sqrt{n}} \Big| \sum_{d\in {\mathcal E}(\kappa, a)} \chi_d(\ell) \Phi\Big(\frac{\kappa d}{X}\Big) \Big| \ll 
\delta(\ell=\square) \frac{X}{L} + X^{\frac 12+\epsilon} \sqrt{\ell}, 
\end{equation} 
where $\delta(\ell =\square)$ denotes $1$ when $\ell$ is a square, and $0$ otherwise.

Henceforth we restrict attention to the terms in $S_2$ where $(n,N_0)= 1$.  Note that if $d\equiv a \bmod {N_0}$ then $d$ is automatically $1 \bmod 4$, and the condition that 
$d$ is a fundamental discriminant amounts to $d$ being square-free.  We express the square-free condition by M{\" o}bius inversion $\sum_{\alpha^2 |d} \mu(\alpha)$, and then 
split the sum into the cases where $\alpha >A$ is large, and when $\alpha \le A$ is small, for a suitable parameter $A\le X$.  We first handle the case when $\alpha > A$ is large.  These terms give 
\begin{align} 
\label{5.8} 
&\sum_{\alpha> A} \mu(\alpha) \sum_{\substack{ d\equiv a \bmod N_0 \\ \alpha^2 | d} }\Phi\Big(\frac{\kappa d}{X}\Big) \frac{1}{L} \sum_{(n,N_0)=1} \frac{\Lambda_E(n)}{\sqrt{n}} 
\Big( {\widehat h}\Big( \frac{\log n}{L} \Big) + {\widehat h} \Big( -\frac{\log n}{L}\Big) \Big) \chi_d(\ell n) \nonumber \\ 
&\ll \sum_{\alpha >A} \sum_{\substack{ d\equiv a \bmod N_0 \\ \alpha^2 | d} }\Phi\Big(\frac{\kappa d}{X}\Big) (\log X) \ll \frac{X}{N_0 A} \log X, 
\end{align} 
upon using GRH to estimate the sum over $n$ and then estimating the sum over $d$ trivially. 

We are left with the terms with $\alpha \le A$, and writing $d=k\alpha^2$ we may express these terms as 
\begin{equation} 
\label{5.9}
\frac 1L \sum_{(n,N_0)=1} \frac{\Lambda_E(n)}{\sqrt{n}} \Big( {\widehat h}\Big(\frac{\log n}{L}\Big) + {\widehat h}\Big(-\frac{\log n}{L}\Big)\Big)  
\sum_{\substack{\alpha \le A \\ (\alpha, n\ell N_0)=1} } \mu(\alpha) \sum_{\substack{ k \equiv a \overline{\alpha^2} \bmod{N_0} }} \Big( \frac{k}{\ell n}\Big) \Phi \Big( \frac{\kappa k \alpha^2}{X}\Big). 
\end{equation} 
We now apply the Poisson summation formula to the sum over $k$ above, as in Lemma 7 of \cite{RaSo1}.  This transforms the sum over $k$ above to 
\begin{equation} 
\label{5.10} 
\frac{X}{n\ell N_0 \alpha^2} \Big(\frac{N_0}{n\ell}\Big) \sum_{v} e\Big( \frac{va\overline{\alpha^2 n\ell}}{N_0}\Big) \tau_v(n\ell) {\widehat \Phi}\Big( \frac{Xv}{n\ell \alpha^2 N_0}\Big), 
\end{equation} 
where $\tau_v(n\ell)$ is a Gauss sum given by 
$$ 
\tau_v(n\ell) = \sum_{b\bmod {n\ell}} \Big(\frac{b}{n\ell}\Big) e\Big( \frac{v b}{n\ell} \Big). 
$$ 
The Gauss sum $\tau_v(n\ell)$ can be described explicitly, see Lemma 6 of \cite{RaSo1} which gives an evaluation of 
$$
G_v(n\ell) = \Big( \frac{1-i}{2} + \Big(\frac{-1}{n\ell}\Big) \frac{1+i}{2}\Big) \tau_v(n\ell), 
$$ 
from which $\tau_v(n\ell)$ may be obtained via 
\begin{equation} 
\label{5.10a} 
\tau_v(n\ell) =  \Big( \frac{1+i}{2} + \Big(\frac{-1}{n\ell}\Big) \frac{1-i}{2}\Big) G_v(n\ell).
\end{equation}  

The term $v=0$ in \eqref{5.10} leads to a main term; we postpone its treatment, and first consider the contribution of terms $v\neq 0$.  Since 
${\widehat h}$ is supported in $[-1,1]$, we may suppose that $n\le e^L$.  The rapid decay of the Fourier transform ${\widehat \Phi}(\xi)$ allows 
us to restrict attention to the range $|v| \le \ell e^L A^2 X^{-1+\epsilon}$, with the total contribution to $S_2$ of terms with larger $|v|$ being estimated by $O(1)$.  
For the smaller values of $v$, we interchange the sums over $v$, performing first the sum over $n$ using GRH.  Thus these terms contribute 
\begin{align*} 
\frac{X}{\ell L N_0 }& \sum_{0< |v| \le \ell e^{L}A^2 X^{-1+\epsilon}} \sum_{\substack{\alpha \le A \\ (\alpha, \ell N_0)=1}} \frac{\mu(\alpha)}{\alpha^2} \\
&\sum_{(n,\alpha N_0)=1} \frac{\Lambda_E(n)}{n\sqrt{n}} \Big(\frac{N_0}{n\ell}\Big) e\Big( \frac{va \overline{\alpha^2n\ell}}{N_0}\Big) \tau_v(n\ell) 
 \Big( {\widehat h}\Big(\frac{\log n}{L}\Big) + {\widehat h}\Big(-\frac{\log n}{L}\Big)\Big)  {\widehat \Phi}\Big( \frac{Xv}{n\ell \alpha^2 N_0}\Big). 
\end{align*} 
We now claim that (on GRH) the sum over $n$ above is 
\begin{equation} 
\label{5.11} 
\ll \frac{\alpha \ell^{\frac 32}}{\sqrt{X|v|}} X^{\epsilon}, 
\end{equation} 
so that the contribution of the terms with $v\neq 0$ is 
\begin{equation} 
\label{5.12} 
\ll X^{\frac 12 +\epsilon} \ell^{\frac 12} \sum_{1\le |v| \le  \ell e^{L}A^2 X^{-1+\epsilon}} |v|^{-\frac 12} \log A \ll \ell e^{L/2} A X^{\epsilon}.
\end{equation} 
To minimize the combined contributions of the error terms in \eqref{5.12} and \eqref{5.8}, we shall choose $A= (X/\ell)^{\frac 12} e^{-\frac L4}$, so that the effect of 
both these error terms is 
\begin{equation} 
\label{5.13} 
\ll X^{\frac 12+\epsilon} \ell^{\frac 12} e^{\frac L4}.
\end{equation} 

To justify the claim \eqref{5.11} we first use \eqref{5.10a} to replace $\tau_v(n\ell)$ by $G_v(n\ell)$ so that we must bound (for both choices of $\pm$)
$$ 
\sum_{(n,\alpha N_0)=1} \frac{\Lambda_E(n)}{n\sqrt{n}} \Big(\frac{\pm N_0}{n\ell}\Big) e\Big( \frac{va \overline{\alpha^2n\ell}}{N_0}\Big) G_v(n\ell) 
 \Big( {\widehat h}\Big(\frac{\log n}{L}\Big) + {\widehat h}\Big(-\frac{\log n}{L}\Big)\Big)  {\widehat \Phi}\Big( \frac{Xv}{n\ell \alpha^2 N_0}\Big).
 $$
 First consider the generic case when $n$ is a prime power with $(n,v)=1$.  Here (using Lemma 6 of \cite{RaSo1}) $G_v(n\ell) =0$ unless $n$ is a prime $p$ not dividing $\ell$ in which case $G_v(p\ell) =  (\frac{v}{p}) \sqrt{p} G_v(\ell)$.  
 Thus such terms contribute to the above 
 $$ 
 \Big( \frac{\pm N_0}{\ell}\Big) G_v(\ell) \sum_{p\nmid \alpha v \ell N_0} \frac{\Lambda_E(p)}{p} \Big(\frac{\pm vN_0}{p}\Big)  e\Big( \frac{va \overline{\alpha^2p\ell}}{N_0}\Big) \Big( {\widehat h}\Big(\frac{\log p}{L}\Big) + {\widehat h}\Big(-\frac{\log p}{L}\Big)\Big)  {\widehat \Phi}\Big( \frac{Xv}{p\ell \alpha^2 N_0}\Big)
 $$
The rapid decay of ${\widehat \Phi}(\xi)$ implies that we may restrict attention above to the range $p> X^{1-\epsilon} |v|/(\ell \alpha^2 N_0)$.  Then splitting $p$ into 
progressions $\bmod N_0$ and using GRH (it is here that we need GRH for twists of $L(s,E)$ by quadratic characters, as well as all Dirichlet characters modulo $N_0$) we obtain the 
bound 
$$ 
\ll |G_v(\ell)| \frac{X^{\epsilon} \ell^{\frac 12} \alpha N_0}{\sqrt{X |v|}} \ll \frac{\ell^{\frac 32} \alpha X^{\epsilon}}{\sqrt{X |v|}},  
$$  
which is in keeping with \eqref{5.11}.   Now consider the non-generic case when $n$ is the power of some prime dividing $v$.  We may assume that $n|v^2$ (else $G_v(n\ell)=0$ by Lemma 6 of \cite{RaSo1}) and also that 
$n\ge X^{1-\epsilon} |v|/(\ell \alpha^2 N_0)$ else the Fourier transform $\widehat \Phi$ is negligible.  Using that $|G_v(n\ell)| \le (v,n\ell)^{\frac 12} (n\ell)^{\frac 12} \le (|v| n\ell)^{\frac 12}$ (which again follows from Lemma 6 of \cite{RaSo1}) we may bound the contribution of these terms by 
$$ 
\ll \sum_{n|v^2} \Lambda(n) \frac{(|v|\ell)^{\frac 12}}{(X^{1-\epsilon} v/(\ell \alpha^2 N_0))} \ll (\log v) X^{\epsilon} \frac{\ell^{\frac 32} \alpha^2}{X \sqrt{|v|}} 
\ll  \frac{\ell^{\frac 32} \alpha X^{\epsilon}}{\sqrt{X |v|}},
$$
since $\log v \ll \log X \ll X^{\epsilon}$ and $\alpha \le A \le \sqrt{X}$.  Thus these terms also satisfy the claimed bound \eqref{5.11}.

Now we handle the main term contribution from $v=0$, noting that $\tau_0(n\ell) =0$ unless $n\ell$ is a square, in which case it equals $\phi(n\ell)$.  Thus the 
main term contribution from $v=0$ is 
$$ 
\frac{X}{L N_0} \sum_{\substack{(n,N_0)=1 \\ n\ell = \square}} \frac{\Lambda_E(n)}{\sqrt{n}} \frac{\phi(n\ell)}{n\ell} \Big(\sum_{\substack {\alpha \le A \\ (\alpha, \ell n N_0)=1}}
\frac{\mu(\alpha)}{\alpha^2}\Big) {\widehat \Phi}(0)  \Big( {\widehat h}\Big(\frac{\log n}{L}\Big) + {\widehat h}\Big(-\frac{\log n}{L}\Big)\Big).
$$
Thus this main term only exists if $\ell$ is a square (so that $n$ is a square), or if $\ell$ is $q$ times a square for a unique prime $q$ (so that $n$ is an odd 
power of $q$).   In the case $\ell$ is a square, writing $n=m^2$ and performing the sum over $\alpha$,  we obtain that the main term is 
$$ 
\frac{X}{L N_0} {\widehat \Phi}(0) \sum_{(m,N_0)=1} \frac{\Lambda_E(m^2)}{m}\Big( \prod_{p|m\ell} \Big(1+\frac 1p\Big)^{-1}  \prod_{p\nmid N_0} \Big(1-\frac 1{p^2}\Big) + 
O\Big( \frac{1}{A}\Big) \Big) \Big( {\widehat h}\Big(\frac{2\log m}{L}\Big) + {\widehat h}\Big(-\frac{2\log m}{L}\Big)\Big).
$$
Using \eqref{4.3} and partial summation we conclude that the main term when $\ell$ is a square is 
\begin{align} 
\label{5.15} 
-\frac{X}{LN_0} {\widehat \Phi}(0)& \Big( \prod_{p|\ell} \Big( 1 + \frac 1p\Big)^{-1}  \prod_{p\nmid N_0} \Big( 1-\frac 1{p^2}\Big) + O\Big(\frac 1A\Big)\Big) 
\Big( \int_{1}^{\infty} \Big( \widehat{h}\Big( \frac{2\log y}{L}\Big) + {\widehat h} \Big( -\frac{2\log y}{L}\Big) \Big) \frac{dy}{y} + O(1)\Big) \nonumber\\ 
=&- \frac{X}{N_0} {\widehat \Phi}(0) \frac{{ h}(0)}{2} \prod_{p|\ell} \Big( 1 + \frac 1p\Big)^{-1}  \prod_{p\nmid N_0} \Big( 1-\frac 1{p^2}\Big) + 
O\Big( \frac{X}{A} + \frac{X}{L}\Big). 
\end{align} 
Suppose now that $\ell$ is $q$ times a square, for a (unique) prime $q$.  Here the main term may be bounded by 
\begin{equation} 
\label{5.16} 
\ll \frac{X}{LN_0} \frac{\log q}{\sqrt{q}} \prod_{p|\ell} \Big(1 +\frac 1p \Big)^{-1} \prod_{p\nmid N_0} \Big(1 -\frac 1{p^2}\Big); 
\end{equation} 
naturally we can be more precise here, but this bound suffices.

   \section{Proof of Proposition \ref{prop3}}

The $k$-th moment in \eqref{2.2} is treated in Proposition 6 of \cite{RaSo1}.   Briefly, expanding out ${\mathcal P}(d;x)^k$ we must handle 
$$ 
\sum_{\substack{p_1, \ldots, p_k \le x \\ p_i \nmid N_0}} \frac{a(p_1) \cdots a(p_k)}{\sqrt{p_1 \cdots p_k}} \sum_{d\in {\mathcal E}(\kappa, a)} \chi_d(p_1\cdots p_k) 
\Phi \Big( \frac{\kappa d}{X}\Big).  
$$
When $p_1\cdots p_k$ is not a perfect square, the sum over $d$ exhibits substantial cancellation (as mentioned earlier in \eqref{5.6}).  The main term arises from 
terms where $p_1 \cdots p_k$ is a perfect square, which cannot happen when $k$ is odd.  When $k$ is even, the contribution to the main term comes essentially from the 
case when there are $k/2$ distinct primes among $p_1$, $\ldots$, $p_k$ with each distinct prime appearing  twice.  The number of such pairings leads to the coefficient $M_k$, and 
Rankin-Selberg theory is used to obtain $\sum_{p\le x} a(p)^2/p = \log \log x +O(1) \sim \log \log X$.

  To establish \eqref{2.3}, once again we expand ${\mathcal P}(d;x)^k$ and are faced with evaluating 
  $$ 
  \sum_{\substack{p_1, \ldots, p_k \le x \\ p_i \nmid N_0}} \frac{a(p_1) \cdots a(p_k)}{\sqrt{p_1 \cdots p_k}} \sum_{d\in {\mathcal E}(\kappa, a)} \chi_d(p_1\cdots p_k) 
  \Big( \sum_{\gamma_d} h\Big( \frac{\gamma_d L}{2\pi} \Big) \Big).
  $$
  We now appeal to Proposition \ref{prop2}.  The terms where $p_1\cdots p_k$ is neither a square nor a prime times a square contribute, using \eqref{2.11}, 
  $$ 
  \ll X^{\frac 12+\epsilon} e^{\frac L4} \sum_{p_1, \ldots, p_k \le x} 1 \ll X^{\frac 12 +\epsilon} e^{\frac L4}. 
  $$ 
  It remains to consider the cases when this product is a square (which can only happen when $k$ is even) and when it is a prime times a square (which can only happen for odd $k$).  
  In the first case, we obtain (by \eqref{2.12}) a main term 
  $$ 
  \frac{X}{N_0} \prod_{p\nmid N_0} \Big( 1-\frac{1}{p^2}\Big) 
  {\widehat \Phi}(0) \Big( \frac{2\log X}{L} {\widehat h}(0)+ \frac{h(0)}{2}+O\Big( \frac 1L\Big)\Big) \sum_{\substack{p_1, \ldots, p_k \le x \\ p_i \nmid N_0 \\ p_1 \cdots p_k =\square} } 
  \frac{a(p_1)\cdots a(p_k)}{\sqrt{p_1\cdots p_k}}\prod_{p|p_1\cdots p_k} \Big(1+\frac 1{p}\Big)^{-1}. 
  $$
  As before, this main term is dominated by the contribution of terms where there are $k/2$ distinct primes among $p_1$, $\ldots$, $p_k$ each appearing twice, and thus we obtain 
  $$ 
  \frac{X}{N_0} \prod_{p\nmid N_0} \Big( 1-\frac{1}{p^2}\Big) 
  {\widehat \Phi}(0) \Big( \frac{2\log X}{L} {\widehat h}(0)+ \frac{h(0)}{2}+O\Big( \frac 1L\Big)\Big) (M_k +o(1)) (\log \log X)^{\frac k2}. 
  $$ 
  This establishes the result \eqref{2.3} for the case $k$ even.  When $k$ is odd, the contribution of the terms when $p_1\cdots p_k$ is a prime times a square 
  may be bounded by (using \eqref{2.13} of Proposition \ref{prop2}) 
  $$
  \ll \frac{X}{LN_0} \sum_{q\le x} \frac{\log q}{q} \Big(\sum_{\substack{p\le x \\ p\nmid N_0}} \frac{a(p)^2}{p}\Big)^{\frac{k-1}{2}} \ll \frac{X}{N_0} \frac{\log x}{L} (\log \log X)^{\frac{k-1}{2}}
  \ll \frac{X}{N_0} (\log \log X)^{\frac{k-1}{2}}, 
  $$ 
which establishes \eqref{2.3} since $M_k=0$ here.

\bibliographystyle{plain} 
\bibliography{Refs} 

\end{document}